\documentclass[a4paper,reqno]{amsart}
 \usepackage{amssymb,amsmath,amscd}
\usepackage[usenames,dvipsnames]{color}
\usepackage{epsf}
 \usepackage{leftidx}
\usepackage{tikz}
\usepackage[all]{xy}
\usepackage{pstricks}
 \usepackage{combelow}
\usepackage{pdfpages}

\newcommand{\id}{\operatorname{id}}

\newcommand{\diffto}{\xrightarrow{\raisebox{-0.2 em}[0pt][0pt]{\smash{\ensuremath{\sim}}}}}

\newcommand{\rmap}{\longrightarrow}

\newcommand{\inc}{\hookrightarrow}

\newcommand{\cD}{\mathbb{D}}

\newcommand{\bd}{\partial}

\newcommand{\trans}{\;\;\makebox[0pt]{$\top$}\makebox[0pt]{$\cap$}\;\;}

\newcommand{\symb}{\operatorname{symb}}

\newcommand{\dom}{\operatorname{dom}}

\newcommand{\Symp}{\operatorname{Symp}}
\newcommand{\Poiss}{\operatorname{Poiss}}
\newcommand{\Diff}{\operatorname{Diff}}

\newcommand{\X}{\mathfrak{X}}

\newcommand{\e}{\varepsilon}

\newcommand{\p}{\pi}
\newcommand{\w}{\omega}
\newcommand{\s}{\mathbb{S}}
\newcommand{\W}{\Omega}
\renewcommand{\th}{\theta}

\newcommand{\R}{\mathbb{R}}
\newcommand{\C}{\mathbb{C}}

\newcommand{\Z}{\mathbb{Z}}
\newcommand{\inco}{\operatorname{in}}
\newcommand{\out}{\operatorname{out}}

\newtheorem{lemma}{{\bf Lemma}}[section]
\newtheorem{theorem}{{\bf Theorem}}[section]
\newtheorem*{theorem1}{Theorem A}
\newtheorem*{theorem2}{Theorem B}
\newtheorem*{theorem3}{Theorem C}
\newtheorem{corollary}{{\bf Corollary}}[section]
\newtheorem{definition}{ {\bf Definition}}[section]
\newtheorem{example}{Example}[section]
\newtheorem{remark}{Remark}[section]

\newtheorem*{acknowledgements}{Acknowledgements}

\begin{document}

  \author{Pedro Frejlich}
\address{ Departamento de Matem\'{a}tica
PUC Rio de Janeiro \\ Rua Marqu\^{e}s de S\~{a}o Vicente, 225 \\ G\'{a}vea, Rio de Janeiro - RJ, 22451-900
\\ Brazil} \email{frejlich.math@gmail.com}\author{David Mart\'inez Torres} \address{ Departamento de Matem\'{a}tica
PUC Rio de Janeiro \\ Rua Marqu\^{e}s de S\~{a}o Vicente, 225 \\ G\'{a}vea, Rio de Janeiro - RJ, 22451-900
\\ Brazil}\email{dfmtorres@gmail.com}

\author{Eva Miranda}\address{Department of Mathematics, Universitat Polit\`{e}cnica de Catalunya and BGSMath \\ EPSEB, Edifici P, Avinguda del Doctor Mara\~{n}\'{o}n, 42-44,
 Barcelona, Spain}
\email{eva.miranda@upc.edu}  \thanks{P. Frejlich has been supported by NWO-Vrije competitie  grant 'Flexibility and Rigidity of Geometric Structures' 612.001.101, and by IMPA (CAPES-FORTAL project).
 D. Mart\'inez Torres has been partially supported by FCT
Portugal (Programa Ci\^encia) and ERC Starting Grant no. 279729. E. Miranda has been partially supported by  the
  project \emph{ Geometr\'{\i}a algebraica, simpl\'{e}ctica, aritm\'{e}tica y aplicaciones}
  with reference: MTM2012-38122-C03-01/FEDER and the Severo Ochoa
   program reference Sev-2011-0087.  The authors are partially supported by European Science Foundation
   network CAST, \emph{Contact and Symplectic Topology}. We would like to thank the CRM-Barcelona  for their hospitality during the Research Programme \emph{Geometry and Dynamics of Integrable Systems}. }

\title{A note on symplectic topology of $b$-manifolds}

\date{\today}
\maketitle
\begin{abstract}
A Poisson manifold $(M^{2n},\p)$ is $b$-symplectic if $\bigwedge^n\p$ is transverse to the zero section.
In this paper we apply techniques native to Symplectic Topology  to address questions pertaining to $b$-symplectic manifolds.
We provide constructions of  $b$-symplectic structures
on open manifolds by Gromov's $h$-principle, and of $b$-symplectic manifolds with a prescribed singular locus, by means of surgeries.
\end{abstract}

\section{Introduction and statement of main results}

A Poisson structure on a manifold $M$ can be described as a bivector $\pi \in \X^2(M)$ which obeys the partial differential
equation $[\pi,\pi]=0$, where $[\cdot,\cdot]$ is the Schouten bracket of multivector fields. The image of the induced bundle
map $\pi^{\sharp}:T^{\ast}M\rightarrow TM$ is an involutive distribution, of possibly varying rank, each of
whose integral submanifolds carries an induced symplectic form.

Symplectic structures are those Poisson structures whose underlying foliation has $M^{2n}$ as its only leaf; equivalently,
they are those Poisson structures for which $\pi^{\sharp}$ is invertible -- i.e., for which $\bigwedge^n \p$ does not meet the zero section.

In \cite{GMP12}, this nondegeneracy condition has been relaxed in a very natural way, by demanding that $\bigwedge^n\p$ be
\emph{transverse} to the zero section instead of avoiding it:

\begin{definition} A Poisson manifold $(M^{2n},\pi)$ is of {\bf $b$-symplectic type} if $\bigwedge^n\pi$ is transverse
to the zero section $M \subset \bigwedge^{2n}TM$.
\end{definition}

Such structures were first defined, in the case of dimension two, by Radko \cite{Ra02},
who called them \emph{topologically stable} Poisson structures. Poisson structures of $b$-symplectic
type have also appeared under the name \emph{log symplectic} \cite{GuL,Cav13,MT1,MT2}.

Symplectic structures are those Poisson structures of $b$-symplectic type whose
{\bf singular locus} $Z(\p):=(\bigwedge^n\p)^{-1}M \subset M$ is empty. Quite crucially for what follows is that general
Poisson structures of $b$-symplectic type do not stay too far from being symplectic.

The transversality condition $\bigwedge^n\p \trans M$ ensures that the singular locus
$Z=Z(\p)$ is a codimension-one submanifold of $M$, which by the Poisson condition is itself foliated in codimension
one by symplectic leaves of $\p$.  Hence, $Z$ is a corank-one Poisson submanifold. Those vector fields $v \in \X(M)$ which are tangent to $Z$ form the space of
all sections of a vector bundle $\leftidx{^b}{T(M,Z)}{} \to M$, called the {\bf $b$-tangent bundle} \cite{Me93}.
The bundle $\leftidx{^b}{T(M,Z)}{}$ has a canonical structure of Lie algebroid, and a Poisson structure of
$b$-symplectic type $\pi$ on $M$ with
singular locus $Z$ can be described alternatively by  a closed, nondegenerate section of
$\bigwedge^2\left(\leftidx{^b}{T^*(M,Z)}{}\right)$, in complete analogy with the
symplectic case. This viewpoint motivates the nomenclature adopted in \cite{GMP12,NT96}
\footnote{Closed, nondegenerate sections of $\bigwedge^2\left(\leftidx{^b}{T^*(M,Z)}\right)$ were introduced in \cite{NT96}
for  $Z=\partial M$.}, and to which we adhere.  Henceforth, we will refer to Poisson structures of $b$-symplectic type as
{\bf $b$-symplectic} structures. With this perspective, it is not surprising that many tools from Symplectic Topology can be
adapted to this $b$-setting. In fact,
the purpose of this paper is to use such tools to discuss the following existence problems in $b$-symplectic geometry:

\begin{enumerate}
 \item[(P1)] Which manifolds $M$ carry a structure of $b$-symplectic manifold?
 \item[(P2)] Which corank one Poisson manifolds $Z$ appear as singular loci of closed $b$-symplectic manifolds ?
\end{enumerate}

For closed manifolds the answer to $($P1$)$ is unknown, even in the symplectic case. For \emph{open} manifolds (i.e., whose connected components either have non-empty boundary or are non-compact), we show:
 \begin{theorem1}
 An orientable, open manifold $M$ is $b$-symplectic if and only if $M \times \C$ is almost-complex.
\end{theorem1}

In fact, the story here is completely analogous to
the symplectic case: supporting a $b$-symplectic structure imposes restrictions on the de Rham cohomology
of a closed manifold \cite{Cav13,MT1},
but these do not apply to open manifolds. There, we show the existence
of $b$-symplectic structures is a purely homotopical question, which abides by a version of the $h$-principle of Gromov \cite{Gr86}. In some very special cases, the finer control granted by having an $h$-principle description allows one to prescribe the singular locus $Z$ of the ensuing $b$-symplectic manifold. (However, in the case where $Z$ bounds a compact region in $M$, these techniques break down completely.)

The singular locus $Z$ a $b$-symplectic manifold is not just a general corank-one Poisson manifold, in that it can be defined by a {\bf cosymplectic structure} $(\theta,\eta) \in \Omega^1(Z) \times \Omega^2(Z)$ (i.e., a pair of closed forms for which $\theta\wedge \eta^{n-1}$ is a volume form; cf. \cite[Proposition 10 ]{GMP12} or Section \ref{sec : Prescribing the singular locus of a b-symplectic manifolds}).
Thus, the Existence Problem $($P2$)$ is about describing which cosymplectic structures appear
as the singular locus of a closed  $b$-symplectic manifold.

Our first result says that $($P2$)$ can be rephrased as a problem of Symplectic Topology; namely,
that of  determining those (closed) cosymplectic manifolds which
admit symplectic fillings (see Section 2).

\begin{lemma}\label{realizable if and only if nullcobordant}  A cosymplectic manifold $(Z,\eta,\th)$ is
the singular locus of a closed, orientable $b$-symplectic manifold if and only if $(Z,\eta,\th)$ is
symplectically fillable.
\end{lemma}

Symplectic fillings of contact manifolds -- and more generally symplectic cobordisms with concave/convex boundaries --
are central to Symplectic Topology, whereas the case of cosymplectic (or flat) boundaries has received comparatively little attention. A notable exception is Eliashberg's result that $3$-dimensional {\bf symplectic mapping tori} (i.e., suspensions of a symplectomorphisms of surfaces) are symplectically fillable \cite{E04}.

Our second result follows from observing that symplectic fillability of all cosymplectic 3-manifolds is a consequence
of symplectic fillability of all symplectic mapping tori. This solves the cosymplectic existence problem in dimension 3:

\begin{theorem2}
Any cosymplectic manifold of dimension $3$ is the singular locus of orientable, closed, $b$-symplectic manifolds.
\end{theorem2}

Our third result describes a class of symplectomorphisms $\varphi$  which yield symplectically
fillable symplectic mapping tori in arbitrary dimensions:  namely, those built out of {\bf Dehn twists}  around parametrized Lagrangian spheres
(see Definition \ref{definition : dehn twist}) and their inverses:

\begin{theorem3}
If $Z$ is a symplectic mapping torus defined by a symplectomorphism which is Hamiltonian isotopic to
a word on Dehn twists and their inverses, then $Z$ is the singular locus of orientable, closed, $b$-symplectic
manifolds.
\end{theorem3}

\begin{remark} Symplectic fillability of a sympletic mapping torus is a property that only depends on the Hamiltonian isotopy class of the symplectomorphism.
 The key result in Eliashberg's argument in \cite{E04}
is that, for surfaces, the fillability of a symplectic mapping torus depends  just on the symplectic isotopy class of the symplectomorphism.
We do not know whether this is true
in higher dimensions.
\end{remark}

While this project was being completed the authors learned of research by G. Cavalcanti which
has some overlap with theirs. More precisely,
the idea of constructing $b$-symplectic manifolds  without boundary by gluing cosymplectic cobordisms
appeared independently in \cite{Cav13}.

 \begin{acknowledgements}  We would like to thank M. Crainic, R. Loja Fernandes, I. M\u{a}rcu\cb{t}, Y. Mitsumatsu,
  A. Mori, B. Osorno Torres, F. Presas and  G. Scott for useful conversations.
 \end{acknowledgements}

\section{Cosymplectic cobordisms and $b$-symplectic structures}\label{sec:cosymp-cob}

We summarize below basic facts and conventions about $b$-symplectic manifolds and cosymplectic
cobordisms, and describe the relation between both structures. For a more detailed account we refer the reader to \cite{GuL,GMP11,GMP12,MT2,Me93,NT96}.

\subsection{$b$-manifolds}

The Lie subalgebra $\X(M,Z)\subset \X(M)$ consisting of those vector fields $v$ which are tangent to $Z$ can be identified with the space of smooth sections of the {\bf $b$-tangent bundle} $\leftidx{^b}{T(M,Z)}{} \to M$.
 By its very construction, $\leftidx{^b}{T(M,Z)}{}$ comes equipped with a bundle map
 $\leftidx{^b}{T(M,Z)}{} \to TM$ covering $\id_M$, which is the identity outside $Z$.
 Its restriction to $Z$ defines an epimorphism $\leftidx{^b}{T(M,Z)}{}|_Z \to TZ$, whose kernel $\leftidx{^b}{N(M,Z)}{}$ has a canonical trivialization
 $\nu \in \Gamma(Z,\leftidx{^b}{N(M,Z)})$: if one expresses $Z$ locally as $x_1=0$ in a coordinate chart $(x_1,...,x_n)$,
 then $x_1\frac{\bd}{\bd x_1}$ is a no-where vanishing local section of $\leftidx{^b}{T(M,Z)}{}$ independent of choices along $Z$.

The bundle dual to $\leftidx{^b}{T(M,Z)}{}$ will be denoted by $\leftidx{^b}{T^{\ast}(M,Z)}{}$;
sections of its $p$-th exterior power will be called $b$-forms
(of degree $p$) on $(M,Z)$, and we write $\leftidx{^b}{\Omega^p(M,Z)}{}$ for the space of all such forms.

Since $\X(M,Z) \subset \X(M)$ is a Lie subalgebra, $\leftidx{^b}{T(M,Z)}{}$ has a natural structure of
Lie algebroid, and as such, it carries a differential
\[\leftidx{^b}{d}{}:\leftidx{^b}{\Omega^p(M,Z)}{} \to \leftidx{^b}{\Omega^{p+1}(M,Z)}{}\]
given by the usual Koszul-type formula.
Note that  agrees with $d$ outside $Z$, and that we have a short exact sequence of chain complexes:
\begin{equation}\label{diag : canonical exact sequence}
 0 \rmap (\Omega^{\bullet}(M),d) \rmap (\leftidx{^b}{\Omega^{\bullet}(M,Z)}{},\leftidx{^b}{d}{})
 \stackrel{\flat}{\rmap}(\Omega^{\bullet-1}(Z),d) \rmap 0,
 \end{equation}
where  $\flat$ maps a $b$-form $\w$ to its contraction with the canonical
$\nu$.

\subsection{Cosymplectic and $b$-symplectic structures}\label{subsec : b- and co-symplectic structures}

Mimicking the usual terminology,  $\omega\in \leftidx{^b}{\Omega^{2}(M,Z)}{}$  will be called  {\bf $b$-symplectic} if
$\omega^\mathrm{top}$ is nowhere vanishing and $\w$ is closed, $\leftidx{^b}{d}{}\omega=0$.

We recall from \cite[Proposition~20]{GMP12} that there is a bijective correspondence between $b$-symplectic forms on $(M,Z)$, and Poisson
 structures of $b$-symplectic type with singular locus $Z$. We can thus speak unambiguously of \emph{$b$-symplectic manifolds}; that is, if $(M,\pi)$ is
a Poisson manifold of $b$-symplectic type, then $(\p|_{M \diagdown Z(\pi)})^{-1}$ extends to a $b$-symplectic form $\w$ on the $b$-manifold $(M,Z(\pi))$.

Let then $(M,\w)$ be a $b$-symplectic manifold, with singular locus $Z \subset M$.
By the very definition, $\w|_{M \diagdown Z}$ is a symplectic manifold in the usual sense.

Before we describe the Poisson structure around points in the singular locus $Z$,
we recall from \cite{GMP11} how \emph{cosymplectic structures} appear in our context.
Recall that a if $(\theta,\eta) \in \W^1(Z) \times \W^2(Z)$ has the property
that $\theta \wedge \eta^{\mathrm{top}-1}$ is a volume form, then to such a pair there
corresponds a pair $(R,\nu) \in \X(Z) \times \X^2(Z)$, where $R$ is the {\emph Reeb} vector field
\[
  \iota_{R}\theta=1, \ \iota_R\eta = 0,\]
  and $\nu$ is characterized by:
  \[ \ \iota_{\theta}\nu = 0, \ \nu^{\sharp}\eta^{\sharp}+\theta \otimes R = \id_{TZ}.
\]Then:
\[
 d\theta = 0, \ d\eta =0 \ \ \ \iff \ \ [R,\nu]=0, \ [\nu,\nu]=0.
\]Hence from a Poisson-theoretic perspective, a cosymplectic structure $(\theta,\eta)$
is a pair $(R,\nu)$ consisting of a corank-one Poisson structure $\nu \in \Poiss(Z)$,
together with a Poisson vector field $R \in \X(Z)$ transverse to the leaves of $\nu$.

A natural way cosymplectic structures occur in Symplectic Geometry is as hypersurfaces transverse to symplectic vector fields.
The setting is the following: one is given a hypersurface $X$ in a symplectic manifold $(W,\W)$,
and we assume that on an open neighborhood $U \subset W$ of $X$ there exists a symplectic vector field $v \in \X(U)$ transverse to $X$. Denote by
\[
 \varphi : U \times \R \supset \dom(\varphi) \rmap W, \ \ \frac{d}{dt}\varphi_t(x) = v \circ \varphi_t(x)
\]the local flow of $v$, and consider the induced map:
\[
c : X \times \R \cap \dom(c) \rmap W, \ \ c(x, t) := \varphi_t(x).
\]Then the \emph{adapted collar} $c$ embeds in $W$ the open neighborhood
\[
V_{\e} \subset X \times \R, \quad V_{\e} := \{(x, t) |  -\e(x) < t < \e(x)\}
\]
for some small enough positive function $\e : U \to \R_+$, and $c$ pulls $\W$ back to:
\begin{align}\label{eq : symplectic normal form}
c^*\W = dt \wedge \theta +\eta,
\end{align} where $(\theta,\eta):=(c^*\iota_v\W|_X , c^*\W|_X)$  and we have abused notation and omitted the first projection map
on $V_{\e} \subset X \times \R$.


More to the point, a description similar to (\ref{eq : symplectic normal form}) exists for a $b$-symplectic
structure $\w$ in a neighborhood of its singular locus $Z(\w)$. When $Z(\w)$ is coorientable, it goes as follows:
start with any collar $c:V_{\e}\inc M$ extending $\id_{Z(\w)}$
and regard $c$ as a $b$-embedding $c:(V_{\e},Z(\w)) \inc (M,Z(\w))$. If we let $\theta$ stand for $\flat (c^{\ast}\w)$,
then $\widetilde{\eta}:=c^{\ast}\w-d \log |t| \wedge\theta$ is an honest closed two-form
$\widetilde{\eta} \in \W^2(V_{\e})$; in particular, $\eta:=\widetilde{\eta}|_{Z(\w)}$ makes sense as a closed two-form on $Z(\w)$,
and note that $(\theta,\eta)$ defines a cosymplectic structure on $Z(\w)$.

Define on $V_{\e}$ the $b$-symplectic form
 \begin{align}\label{eq : b-symplectic normal form, form}
  \w_0:=d \log |t| \wedge \theta + \eta
 \end{align}
and observe that the path of $b$-forms $\w_t:=\w_0 + t(c^{\ast}\w-\w_0)$ is $b$-symplectic
 on $V_{\e'}$, for some $0<\e'<\e$. Now, $c^{\ast}\w-\w_0 = d\alpha$, for some $\alpha \in \W^1(V_{\e'})$
 which vanishes along $Z(\w)$. Hence the time-dependent $b$-vector field $v_t$ defined by $\iota_{v_t}\w_t+\alpha=0$
 vanishes along $Z(\w)$, and hence the local flow of $v_t$ is defined up to time one on some $V_{\e''}$ for some $0<\e'' < \e'$,
 and satisfies $\phi_t^{\ast}\w_t=\w_0$.  Therefore, we obtain the \emph{adapted collar}:
 \[
\overline{c}:=c\phi_1 : V_{\e''} \inc M
 \] pulling $\w$ back to $\w_0$.
Note that
 in formula (\ref{eq : b-symplectic normal form, form}) the one-form $\theta$ is
 intrinsically defined, while $\eta$ depends on the collar, this choice not affecting its restriction to the symplectic distribution $\ker\theta$. Observe also that the inverse $\p_0:=\w_0^{-1} \in \Poiss(M)$ is given by:
\begin{align}\label{eq : b-symplectic normal form, form}
  \p_0:=t \frac{\bd}{\bd t}\wedge R + \nu,
 \end{align}where $(R,\nu)$ is the pair determined by $(\theta,\eta)$. In particular, on its singular locus $Z(\w)$, a $b$-symplectic form $\w$ determines
 a corank-one Poisson structure $\nu$ which comes from a cosymplectic structure $(\theta,\eta)$ on $Z(\w)$. This is still the case when $Z(\w)$ is not coorientable in $M$ (as in the example below), where a $\mathbb{Z}_2$-equivariant version of the normal form (\ref{eq : b-symplectic normal form, form}) can be proved for the pullback of $\w$ to the orientation covering of $M$.

\begin{example}[Radko's sphere]\label{ex : radko sphere and projective plane} The $b$-form $\omega= \frac{1}{h} dh\wedge d\theta$ on $\s^2$,
where $h,\theta$ stand for cylindrical coordinates, is $b$-symplectic. Its symplectic leaves are
either points in the equator $\s^1 \subset \s^2$, or components of $\s^2 \diagdown \s^1$. This $b$-form is invariant
under the antipodal map on the sphere and, thus, induces a $b$-symplectic form on the projective plane $\mathbb{RP}^2$ which is a non-orientable $b$-symplectic manifold.
\end{example}
\begin{figure}[h!]
\centering
\includegraphics[scale=0.5]{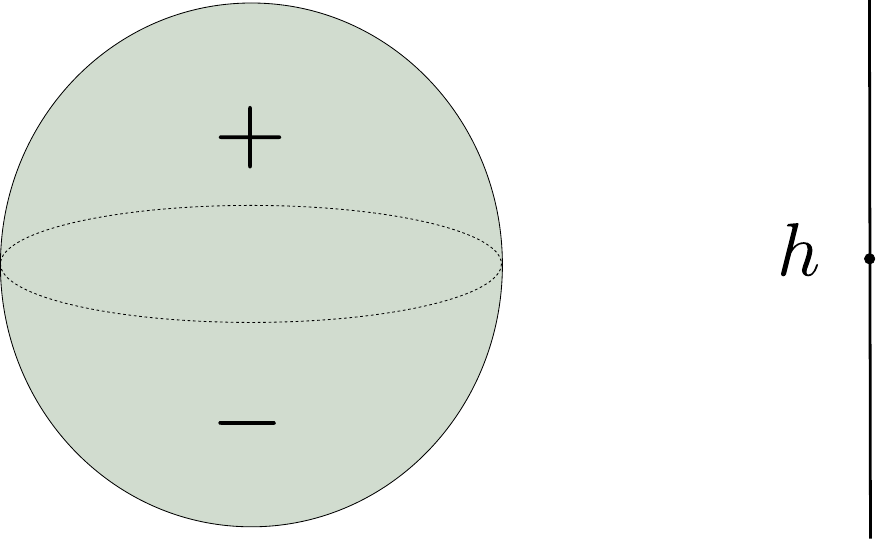}
\caption[]{Radko sphere $\s^2$ with the equator as critical hypersurface and the upper and lower
hemisphere as positive and negative symplectic
leaf, respectively.}
\label{fig:radkosphere}
\end{figure}

\subsection{Cobordisms}
An alternative perspective on $b$-symplectic manifolds is that they arise by gluing certain cobordisms in the symplectic category.

\begin{definition}
 A {\bf cosymplectic cobordism} is a compact symplectic manifold $(M,\w)$, together with a cosymplectic structure of the form $(\theta,\w|_{\bd M})$ on
 its boundary $\partial M$.
\end{definition}
As pointed out in the previous subsection, the cosymplectic structure $(\theta,\w|_{\bd M})$ is induced by a symplectic vector
field $v$ defined in a neighborhood of $\bd M$ and transverse to it.
A connected component $X$ of $\bd M$ is called \emph{incoming} or \emph{outgoing}
acoording to whether a such defining transverse symplectic vector field $v$ points into $M$ or out of it, respectively.

Observe that if $(M_0,\w_0,\theta_0)$ and $(M_1,\w_1,\theta_1)$ are cosymplectic cobordisms, and there exists a diffeomorphism:
\[
 \varphi : \bd M_0 \diffto \bd M_1, \quad \varphi^{\ast}\theta_1 = \theta_0, \ \ \varphi^{\ast}(\w|_{\bd M_1})=\w|_{\bd M_0},
\]then a $b$-symplectic manifold $(M_{01},\w_{01})$ and embeddings $\varphi:M_i \inc M$ exist, with:
\begin{itemize}
 \item $Z(\w_{01}) = \varnothing$ and
 \[
\varphi_0M_0 \cup \varphi_1M_1 = M_{01}
 \]if $\bd M_0$ is incoming and $\bd M_1$ is outgoing, or $\bd M_0$ is outgoing and $\bd M_1$ is incoming;
 \item $Z(\w_{01}) \simeq \bd M_0 \simeq \bd M_1$ and outside a neighborhood $U \simeq Z(\w_{01}) \times (-1,1)$ of $Z(\w_{01})$, we have:
 \[
 \varphi_0M_0 \coprod \varphi_1M_1 = M \diagdown U
 \]if $\bd M_0,\bd M_1$ are both incoming or both outgoing.
\end{itemize}In the first case one constructs collars as in (\ref{eq : b-symplectic normal form, form}) and glues the result into a symplectic manifold; in the second case the orientations do not match, and so one mediates the previous gluing by a collar of the form:
\[
 (Z(\w) \times [-1,1],df \wedge \theta + \eta),
\]where $f=f(t)$ is a monotone function with $df=d\log |t|$ around $t=0$ and $df=\pm dt$ around $t=\pm 1$.

Recall that the {\bf double} $M\# \overline{M}$ of an orientable manifold with boundary can be given, up to isomorphism, a unique smooth structure in which $M \inc M\# \overline{M}$ and $\overline{M} \inc M\# \overline{M}$ are smooth embeddings. An illustration of the natural way $b$-symplectic structures occur when trying to pass from compact symplectic manifolds to closed ones is given by the following result:

\begin{corollary}\label{cor:double}
 If $(M,\w,\theta)$ is a cosymplectic cobordism, then its double $M\# \overline{M}$ is a $b$-symplectic manifold.
\end{corollary}
Of course, when the boundary orientations are opposite, the singular locus of the ensuing $b$-symplectic manifold is empty, and it is a consequence of Moser' s argument that its symplectomorphism type does not depend on the choice of collars.

For example, Radko's sphere is the \emph{b-symplectic double} of any closed symplectic disk $(\mathbb{D}^2,\w)$
endowed with a nowhere zero 1-form on $\mathbb{S}^1$
of period $2\pi$.

\begin{remark}\label{rem:uniqueness}
Just as in the symplectic case, a $b$-symplectic version of Moser's argument \cite{GMP12} shows that the $b$-symplectic type of a $b$-symplectic manifold obtained by gluing symplectic cobordisms with equal boundary orientations does not depend on the choices of collars.
\end{remark}

\section{$h$-principle}

In this section, we use standard $h$-principle arguments to provide a complete answer to the Existence
Problem $($P1$)$ under the additional assumption that $M$ is open.

A necessary condition for a manifold $M^{2n}$ to be symplectic is that it carry a nondegenerate
two-form, or, equivalently, an almost-complex structure. If $M$ is compact, we have a further necessary condition,
namely, that there be a degree-two cohomology class $\tau \in H^2(M)$ with $\tau^n \neq 0$.

For \emph{open} manifolds $M$ a classical theorem of Gromov \cite{Gr86} states that the sole obstruction to the existence of a symplectic structure is that $M$ be almost-complex. More precisely, given any non-degenerate two-form $\w_0$ and a cohomology class $\tau \in H^2(M)$, there is a path $\w :[0,1] \to \Omega^2(M)$ of nondegenerate two-forms connecting $\w_0$ to a closed two-form $\w_1$ representing $[\w_1] = \tau$.

We consider now the case of $b$-symplectic structures. Recall that $b$-symplectic manifolds need \emph{not}
be oriented as usual manifolds, so in particular they may fail to be almost-complex. However:

\begin{lemma}
 If an orientable $M$ admits a $b$-symplectic structure $\w$, then $M \times \C$ is almost-complex.
\end{lemma}The proof is a straightforward adaptation of the argument in \cite[\S 4]{CanGuiWoo}, which we omit.

Just as in the symplectic case, if we demand that $M$ be compact, the existence of a $b$-symplectic structure is obstructed:
\begin{enumerate}
 \item there exists a cohomology
 class $\tau\in H^2(M)$ with $\tau^{n-1} \neq 0$  \cite{MT1};
 \item furthermore, if $M$ is  orientable, a non-trivial $\vartheta \in H^2(M)$
 must exist squaring to zero \cite{Cav13}.
\end{enumerate}
   None of these obstructions appear when $M$ is open, so one wonders if, in that case, $M \times \C$ being
   almost-complex is sufficient to ensure that $M$ carries a $b$-symplectic structure. We answer the question in the affirmative:


\begin{theorem1}\label{thm : open, stably complex is b-symplectic}
 Let $M$ be an orientable, open manifold. Then $M$ is $b$-symplectic if and only if $M \times \C$ is almost-complex.
\end{theorem1}

We need to introduce the analogs of nondegenerate two-forms. To do that, observe that
a bivector $\pi\in\mathfrak{X}^2(M^{2n})$ whose top exterior power $\bigwedge^n\pi$ is transverse to the zero section defines
a $b$-manifold $\left(M,Z(\pi)\right)$, $Z(\pi):=(\bigwedge^n\pi)^{-1}M\subset M$. Since
$b$-bivectors $\X^2(M,Z(\p))$ sit inside the space of all bivectors $\X^2(M)$, it makes sense to require
$\p$ to be a $b$-bivector in the $b$-manifold it defines (a simple
coordinate check shows this is not true in general).

\begin{definition}
  A bivector $\pi \in \X^2(M)$ is {\bf $b$-serious} if it is transversally nondegenerate and a $b$-bivector in $(M,Z(\p))$.
\end{definition}

In the sequel we show that $b$-serious bivectors can be homotoped into $b$-symplectic ones, provided that the manifold be open:

\begin{theorem}\label{homotoping b-serious to poisson}
On an \emph{open} manifold $M$, a $b$-serious bivector $\p_0$ is homotopic through $b$-serious bivectors to a
Poisson bivector $\p_1$. Moreover, one can arrange that $Z(\p_1)$
be non-empty if $Z(\p_0)$ is non-empty.
\end{theorem}
This statement is a result of checking that $1$-jets of Poisson bivectors of exact $b$-symplectic type form a microflexible differential relation, invariant under the pseudogroup of local diffeomorphisms of $M$, cf. \cite{Gr86}. We opted instead to follow the somewhat more visual scheme of proof of \cite{EM}.
\begin{proof}
Take $\p_0 \in \X^2_{\trans}(M)$ $b$-serious, $\p_0 \in \X(M,Z_0)$, and let $\w_0 \in \leftidx{^b}{\Omega^2(M,Z_0)}{}$ the corresponding $b$-symplectic form.

Observe that the $b$-differential $\leftidx{^b}{d}{}:\leftidx{^b}{\Omega^{p}(M,Z_0)}{} \rmap \leftidx{^b}{\Omega^{p+1}(M,Z_0)}{}$
can be factored as a composition $\leftidx{^b}{d}{}=\widetilde{\symb}(\leftidx{^b}{d}{})\circ j_1$, where $j_1$ denotes the $1$-jet map
\begin{gather*}
 j_1:\Gamma(M,\bigwedge^p\leftidx{^b}{T^{\ast}(M,Z_0)}{}) \rmap \Gamma(M,J_1\bigwedge^p\leftidx{^b}{T^{\ast}(M,Z_0)}{})
\end{gather*}
and
\[
\widetilde{\symb}(\leftidx{^b}{d}{}):\Gamma(M,J_1\bigwedge^p\leftidx{^b}{T^{\ast}(M,Z_0)}{}) \rmap \Gamma(M,\bigwedge^{p+1} \leftidx{^b}{T^{\ast}(M,Z_0)}{})
\]is induced by a bundle map
\[
 \symb(\leftidx{^b}{d}{}):J_1\bigwedge^p \leftidx{^b}{T^{\ast}(M,Z_0)}{} \rmap \bigwedge^{p+1}\leftidx{^b}{T^{\ast}(M,Z_0)}{}.
\]
As one easily checks, $\symb(\leftidx{^b}{d}{})$ is an epimorphism with contractible fibres; in particular,
we can lift $\w_0$ to $\widetilde{\w}_0 \in \Gamma(M,J_1\overset{p}{\bigwedge}\leftidx{^b}{T^{\ast}(M,Z_0)}{})$.

Now, since $M$ is an open manifold, there exists a a subcomplex $K$ of a smooth triangulation of $M$, of positive codimension,
with the property that, for an arbitrarily small open $U \subset M$ around $K$, there exists an isotopy of open
embeddings $g_t:M \inc M$, $g_0=\id_M$, with $g_1(M) \subset U$ and $g_t|_K=\id_K$. We will refer to $K$ as a core of $M$,
and say that $g_t$  compresses $M$ into $U$. Note in passing that one can always find a core $K$ of $M$ meeting $Z_0$.

Fix then a core $K$ of $M$, and a compression of $M$ into an open $U$ around $K$. The Holonomic Approximation theorem of
\cite{EM} then says that we can find
\begin{itemize}
 \item an isotopy $h_t$ of $M$ mapping $K$ into $U$;
 \item an open $V \subset U$ around $h_1(K)$;
 \item a section $\alpha \in \Gamma(V,\leftidx{^b}{T^{\ast}(M,Z_0)}{})$
\end{itemize}such that $j_1\alpha$ is so $C^0$-close to $\widetilde{\w}_0$ that we can find a homotopy
\[
 \widetilde{\w}(t) \in \Gamma(V,J_1\leftidx{^b}{T^{\ast}(M,Z_0)}{}),
\]
connecting $\widetilde{\w}_0|_V$ to $j_1\alpha$, and with $\widetilde{\symb}(\leftidx{^b}{d}{})\widetilde{\w}_t$
nondegenerate $b$-forms on $V$. Moreover, the scheme of proof in \cite{EM} shows that one can require in addition that $h_1(K)$ also meet $Z_0$.

Now regard the compression $g_t$ as a smooth family of $b$-maps\[g_t:(M,Z_t)\rmap (M,Z_0), \quad Z_t:=g_t^{-1}Z_0,\]
and set $\w_1:=\leftidx{^b}{d}{}(g_1^{\ast}\alpha) \in \leftidx{^b}{\Omega^2(M,Z_1)}{}$.
Observe now that $\widehat{\w}^1_t:=g_t^{\ast}\widetilde{\w}_0$ connects $\widetilde{\w}_0$ to $g_1^{\ast}(\widetilde{\w}_0|_V)$,
and $\widehat{\w}^2_t:=g_1^{\ast}\widetilde{\w}(t)$ connects $g_1^{\ast}(\widetilde{\w}_0|_V)$ to a lift of $\w_1$.
Let $\widehat{\w}_t$ denote the concatenation of $\widehat{\w}^1_t$ and $\widehat{\w}^2_t$:
\[
 \widehat{\w}_t:=\begin{cases}
                  \widehat{\w}^1_{2t} & 0\leqslant t \leqslant 1/2,\\
		  \widehat{\w}^2_{2t-1} & 1/2\leqslant t \leqslant 1.
                 \end{cases}
\]Then $t \mapsto \p_t:=\widehat{\w}_t^{-1} \in \X(M,Z_t)$ defines a homotopy of $b$-serious bivectors between $\p_0$ and a Poisson $\p_1$.
\end{proof}

A few remarks are in order:
\begin{itemize}
 \item If $\w_0$ could be $C^0$-approximated by a closed $\w_1 \in \leftidx{^b}{\Omega^2(M,Z_0)}{}$, we would be done.
 However, such an approximation is severely obstructed, in that it would imply that $Z_0$ admits a structure of cosymplectic
 manifold.(For compact $Z$, the existence of a cosymplectic structure implies e.g. that it has non-trivial cohomology in all possible degrees).
 \item We get around this problem by changing the topology of $Z_0$ rather drastically; observe in particular that $Z_1$
 may be disconnected even if $Z_0$ is connected. One should perhaps think of $Z_1$ as $Z_0$ with those places `blown to infinity'
 where $\w_0$ cannot be approximated by closed $b$-forms.
 \item Of course, when $M$ is itself almost-complex, Gromov's theorem allows us to produce an honest symplectic structure.
 \item If in the statement of Theorem \ref{homotoping b-serious to poisson} we further assume that:
 \begin{itemize}
  \item $Z=Z(\p_0)$ is a regular fibre $f^{-1}(0)$ of a proper Morse function $f:M \to \R$, unbounded from above and from below, and
 \item $\w_0$ is already $\leftidx{^b}{d}{}$-closed around $Z$,
 \end{itemize}
then one can impose that the homotopy $\p_t$ above be stationary around $Z$ \cite[Theorem 7.2.4]{EM}.
\end{itemize}
This last comment can be regarded as a sufficient condition to realize a given cosymplectic structure on $Z$
on a \emph{given} manifold $M$:

\begin{corollary}
 Any given cosymplectic structure on the regular fibre $Z$ of a proper Morse function $f:M \to \R$ can be realized
 as the singular locus of a $b$-symplectic structure, provided $f$ is unbounded from above and from below.
\end{corollary}
\begin{proof}[Proof of Theorem A]
 It remains to show that the existence of an almost-complex structure on $M \times \C$ ensures the existence of a
 $b$-serious bivector. But according to \cite{Can10}, the former guarantees the existence of a folded symplectic
 form $\phi \in \Omega^2(M)$, namely, a closed two-form, whose top power is transverse to the zero section,
  thus defining a smooth folding locus $Z=({\phi^n})^{-1}M \subset M$,
 along which $\phi^{n-1}$ does not vanish. For every one-form $\theta \in \W^1(Z)$ satisfying $\theta \wedge \phi|_Z^{n-1} \neq 0$ everywhere, one can construct an open embedding
 \[
  c:C \inc M, \quad c^{\ast}\phi = d\left(\frac{t^2}{2}\theta\right) + \phi|_Z,
 \]following the recipe in \cite[Theorem 1]{CanGuiWoo}, where $C \subset Z \times \R$ is an open neighborhood of $Z$.

Fix a Riemannian metric $g$ on $M$ with $c^{\ast}g:=g_Z+dt^2$, and denote by
 $\p$ the bivector on $M$ which is $g$-dual to $\phi$; then
 \[
  c^{\ast}\pi = t\frac{\bd}{\bd t} \wedge R + \nu,
 \]where
 \[
  R= g^{\flat}(\theta), \quad \nu = g^{\flat}\left( \phi|_Z+t^2/2 d\theta\right)
 \]has no $\frac{\bd}{\bd t}$-component. Therefore, $\p$ is $b$-serious, and has singular locus $Z$.
\end{proof}

\section{Prescribing the singular locus of a $b$-symplectic manifold}\label{sec : Prescribing the singular locus of a b-symplectic manifolds}

The Existence problem $($P2$)$ asks for the description of those corank 1 Poisson manifolds that appear as singular locus of a $b$-symplectic manifold:

\begin{definition} A cosymplectic manifold is the singular locus of a $b$-symplectic manifold $(M,\w)$,
 if it is diffeomorphic to $Z(\w)$, endowed with the cosymplectic structure induced by some adapted collar.
 \end{definition}

Thus, our problem splits into two questions: Firstly, deciding which corank-one Poisson structures come
from cosymplectic structures, a matter discussed in  \cite{GMP11}. Secondly, describing which (connected) cosymplectic manifolds are singular loci of $b$-symplectic manifolds, which is the question we will tackle.

Note that the question is trivial if either:
\begin{itemize}
\item the $b$-symplectic manifold is not required to have empty boundary (e.g., a closed, adapted collar associated to a cosymplectic manifold $(Z,\theta,\eta)$ would provide such a realization), or
\item if we do not require that the cosymplectic manifold be the entire singular locus of $(M,\w)$ (e.g., the double of the adapted collar is a $b$-symplectic manifold without boundary, whose singular locus consists of \emph{two} copies of $(Z,\theta,\eta)$).
\end{itemize}

We shall therefore assume that all cosymplectic manifolds appearing henceforth are compact and connected.

\begin{definition}
 Let $(Z_0,\eta_0,\theta_0)$, $(Z_1,\eta_1,\theta_1)$ be cosymplectic manifolds.
 We say that $Z_0$ is {\bf cosymplectic cobordant} to $Z_1$ if there exists a cosymplectic
 cobordism $(M,\w,\theta)$ and diffeomorphisms of cosymplectic manifolds
\[
\varphi_0 : (\bd_{\inco} M,\w|_{\bd_{\inco} M}, \th)
\diffto (Z_0,\eta_0,\th_0), \quad \varphi_1 : (\bd_{\out} M,\w|_{\bd_{\out} M}, \th) \diffto (Z_1,\eta_1,\th_1).
\]
\end{definition}

A cosymplectic manifold $(Z,\eta,\th)$ will be called {\bf symplectically fillable} if it is cosymplectic cobordant to the empty set.

\begin{lemma}\label{realizable if and only if nullcobordant} A cosymplectic manifold
is the singular locus of a closed, oriented $b$-symplectic manifold
if and only if it is symplectically fillable.
\end{lemma}
\begin{proof}
If $(Z,\theta,\eta)$ is the singular locus of $(M,\w)$, then after removing an open, adapted collar
inducing the cosymplectic structure, we obtain a cosymplectic cobordism with two connected components, each of which is a cosymplectic cobordism from $(Z,\theta,\eta)$ to the empty set.

Conversely, if we have a cosymplectic cobordism from $(Z,\theta,\eta)$ to the empty set, then $(Z,\theta,\eta)$ is the singular locus of the double of the cobordism.
\end{proof}

Since cosymplectic cobordisms can be composed, any cosymplectic structure cobordant to a fillable one is also fillable. Thus, it is important to discuss some basic constructions of cosymplectic cobordisms.

\begin{lemma}\label{lem : cosymp cob}
Two cosymplectic structures $(\eta_0,\theta_0)$, $(\eta_1,\theta_1)$ on $Z$ are cobordant if there is a homotopy
$(\eta_t,\th_t)$ of cosymplectic structures joining them, and $[\eta_t]$ is constant.

In particular, among cosymplectic structures on $Z$ symplectic fillability is stable under rescaling of the 1-form by a positive scalar,
small $C^0$-perturbations of the 1-form, and under
replacement of the 2-form by a cohomologous one.
\end{lemma}
\begin{proof}
Subdivide $[0,1]$ into $0=t_0 < t_1 < \cdots <t_N=1$ so that $\th_{t_{i+1}}|_{\ker \eta_{t}}>0$, for all $t \in [t_i,t_{i+1}]$.
It suffices to show that $(Z,\eta_{t_i},\th_{t_i})$ is cobordant to $(Z,\eta_{t_{i+1}},\th_{t_{i+1}})$ for each $i$,
so we may as well assume that $N=1$. Now, $\th_t \wedge dt+\eta_0$ is then a symplectic form on $M=Z\times [0,1]$
defining a cosymplectic cobordism between $(Z,\eta_0,\th_0)$ and $(Z,\eta_0,\th_1)$. Hence we may assume without
loss of generality $\th_0=\th_1$.

Let $\mathcal{F}$ denote the codimension-two foliation on $M$ which is the product of the foliation of the interval by points and the
cosymplectic foliation on $Z$.
We employ a suitable adaptation of Thurston's trick for $\mathcal{F}$.

Let $\eta_0-\eta_t=d\alpha_t$, and choose $\tilde{\alpha}\in \Omega^1(M)$ such that $\tilde{\alpha}_{\mid Z\times\{t\}}=\alpha_t$. Define
$\w'=\eta_0-d\tilde{\alpha}$
which is symplectic on the leaves of $\mathcal{F}$. Since $M$ is compact, for $K>0$ large enough, the form
\[\w=\w'+K\theta_0\wedge dt\]
is symplectic and restricts to $\eta_i$ on $Z_i$. Hence $(M,\w,\th)$, with $\th_{\mid Z\times\{i\}}=\th_i$, $i=0,1$,
is the desired cobordism from  $(\eta_0,\theta_0)$ to $(\eta_1,\theta_1)$. Reversing the direction of the path $(\eta_t,\th_t)$
gives a cobordism in the opposite direction.
\end{proof}

Next, we discuss symplectic fillability for the $C^0$-dense subset of cosymplectic structures given by 1-forms with rational periods.

\subsection*{Symplectic mapping tori and symplectic fillings}

We shall regard a mapping torus as a foliated bundle with base $\s^1$.
The corresponding holonomy representation is generated by a diffeomorphism $\varphi \in \Diff(F)$. Conversely,
the suspension any such $\varphi$ defines a mapping torus $Z(\varphi)$ with fiber diffeomorphic to $F$.

As $F$ is assumed to be compact, mapping tori can be defined as fibrations $Z\rightarrow \s^1$ with an Ehresmann connection. We shall also identify $\s^1$ with $\R/\Z$.

A {\bf symplectic mapping torus} is a bundle over the circle, whose total space is endowed with a closed
two-form $\eta$ which is symplectic on each fiber. The kernel of $\eta$ defines an Ehresmann connection, and its holonomy $\varphi$ preserves
the symplectic structure of the fiber, i.e., $\varphi\in \Symp(F,\sigma)$. Conversely, the suspension of any
$\varphi\in \Symp(F,\sigma)$ canonically defines a symplectic mapping torus $(Z(\varphi),\eta_\varphi)$.

A symplectic mapping torus becomes a cosymplectic manifold upon the choice of a defining closed 1-form for the fibration; this is equivalent
to the choice of a period $\lambda>0$, as it is convened that the pullback of the oriented generator of $H^1(\s^1;\Z)$ has period 1.

We shall abuse notation and regard a symplectic mapping torus as a cosymplectic manifold $(Z(\varphi),\eta_\varphi,\theta_\varphi)$ by declaring $\theta_\varphi$ to have period 1. Having this convention in mind, Lemma  \ref{lem : cosymp cob} implies that a cosymplectic manifold with compact foliation is symplectically fillable if and only if its associated symplectic mapping torus is symplectically fillable.

Let $\varphi_0,\varphi_1\in \Symp(F,\sigma)$. A symplectic isotopy $\varphi_t$ between $\varphi_0$ and $\varphi_1$ corresponds to an isotopy of symplectic vector fields given by the normalized kernel of the closed 2-forms $\eta_{\varphi_t}$. It is well-known that the isotopy can be modified to a Hamiltonian one iff the vector fields are Hamiltonian, i.e. if the cohomology class $[\eta_{\varphi_t}]$ is independent of $t$.
Hence by Lemma \ref{lem : cosymp cob} the symplectic fillability of $(Z(\varphi),\eta_\varphi)$ only depends on the Hamiltonian isotopy class of $\varphi$. Lastly, since changing the identification of the fiber $(F,\sigma)$ by a symplectomorphism results in conjugating the holonomy by that symplectomorphism, we see that whether $(Z(\varphi),\eta_\varphi)$ is symplectically fillable or not $\varphi$ only depends on the Hamiltonian
isotopy class of its conjugacy class.

In what follows, we discuss a family of Hamiltonian isotopy classes of diffeomorphisms yielding fillable
symplectic mapping tori.

\subsection*{Dehn twists}
There is a class of symplectomorphisms $\varphi$ which we can `cap off': Dehn twists. We briefly recall the
construction of those maps, and refer the reader to \cite{Se03} for further details.

The norm function $\mu:(T^{\ast}\s^{n-1} \diagdown \s^{n-1}) \to \R$, $\mu(\xi)=\Vert \xi \Vert$, associated to the
round metric on $\s^{n-1}$, is the moment map of a Hamiltonian
$\s^1$-action on $(T^{\ast}\s^{n-1} \diagdown \s^{n-1})$. Upon identifying $T^{\ast}\s^{n-1} \subset \R^{n}\times \R^{n}$
as $T^{\ast}\s^{n-1}=\lbrace (u,v):\langle u,v\rangle=0,\Vert u \Vert =1\rbrace$, we can write
\[
 e^{2\p it}\cdot (u,v)=\left(\cos(2\p t)u+\sin(2\p t)v\Vert v \Vert^{-1},\cos(2\p t)v-\sin(2\p t)\Vert v\Vert u\right),
\]Then $e^{\p}\cdot (u,v)=(-u,-v)$ extends by the antipodal map to a symplectomorphism $T^{\ast}\s^{n-1} \to T^{\ast}\s^{n-1}$.

Choose now a function $r:\R \to \R$, satisfying:
 \begin{enumerate}
  \item $r(t)=0$ for $|t|\geqslant C >0$;
  \item $r(t)-r(-t)=t$ for and $|t|\ll 1$,
 \end{enumerate}and let $\phi^t$ denote the flow of the Hamiltonian vector field  of $r(\mu)$.

 Observe that $\phi^{2\p}$ extends to a symplectomorphism $\psi:T^{\ast}\s^{n-1} \to T^{\ast}\s^{n-1}$, supported
 on the compact subspace $T(\e)\subset T^{\ast}\s^{n-1}$ the subspace of cotangent vectors of length $\leqslant C$.
 This is called a {\bf model Dehn twist}.

We can graft this construction onto manifolds using Weinstein's Lagrangian neighborhood theorem. If $l:\s^{n-1} \inc (F,\sigma)$
embeds $\s^{n-1}$ as a Lagrangian sphere, there are neighborhoods
$\s^{n-1} \subset U \subset T^{\ast}\s^{n-1}$ and $l(\s^{n-1}) \subset V \subset F$ and
a symplectomorphism $\varphi:(U,\omega_{\textrm{can}}) \to (V,\w)$ extending $l$. If $\psi$ is a model Dehn twist,
supported inside $U$, we produce a symplectomorphism $\tau :(F,\sigma)\to (F,\sigma)$, supported in $V$, by
\[
 \tau(x):=\begin{cases}
           \varphi\circ\psi\circ\varphi^{-1}(x) & \text{ if }x \in V;\\
           x & \text{ if }x \in F\diagdown V.
          \end{cases}
\]
\begin{definition}\label{definition : dehn twist}
 A symplectomorphism of the form above will be called a {\bf Dehn twist} around $l:=l(\s^{n-1})$, and it will be denoted by $\tau_l$.
\end{definition}
We also recall that any two Dehn twists around a parametrized Lagrangian sphere $l$ are
Hamiltonian isotopic if $n>2$, and symplectically isotopic if $n=2$ \cite{Se03}.

\begin{proof}[Proof of Theorem C]

By Lemma \ref{realizable if and only if nullcobordant}, is it enough to show that $Z(\varphi)$ is symplectically fillable.

Assume that  $l:\s^{n-1} \inc Z(\varphi)$, $n>2$, is a parametrized Lagrangian sphere sitting inside a leaf of $Z(\varphi)$,
and observe that the normal bundle to $l$ in $Z(\varphi)$ is trivial and carries a canonical framing.
The cobordism $M$ obtained from $Z(\varphi)\times [0,1]$ by attaching a $n$-handle along $l\times 1$ carries the
structure of a cosymplectic cobordism, as follows from \cite[Proposition 4]{Ma12},
with $\bd_{\inco}M \simeq Z(\varphi)$ and $\bd_{\out}M \simeq Z(\tau_l^{-1}\varphi)$ according
to \cite[Theorem 3]{Ma12}. We call this the \emph{trace} of a positive Lagrangian surgery along $l$.
Negative Lagrangian surgeries can be similarly defined, by attaching a $n$-handle to $l\times 1$
according to the opposite of the canonical framing, and one obtains a cosymplectic cobordism $M$ with
$\bd_{\inco}M \simeq Z(\varphi)$ and $\bd_{\out}M \simeq Z(\varphi\tau_l)$. Note that if, after performing negative Lagrangian surgery, we modify the identification of $F$ by the Dehn twist, we also obtain  $\bd_{\out}M \simeq Z(\tau_l\varphi)$.

By hypothesis, $\varphi$ is --up to conjugacy by a symplectomorphism-- Hamiltonian isotopic to
$\tau_{l_1}^{\varepsilon_1} \cdots \tau_{l_m}^{\varepsilon_m}$, where  $l_i:\s^{n-1}\inc (F,\sigma)$, are parametrized Lagrangian spheres and ${\varepsilon_i}=\pm 1$,  $i=1,\dots,m$.

By the discussion above,
\[
 Z(\varphi), \quad Z(\tau_{l_1}^{-\varepsilon_1}\varphi), \quad \cdots,
\cdots, Z(\tau_{l_1}^{-\varepsilon_1} \cdots \tau_{l_m}^{-\varepsilon_m}\varphi)=Z(\id_F)
\]are all cosymplectic cobordant, and $Z(\id_F)$ is symplectically fillable, since it bounds
$(F \times \cD^2,\sigma+dy_1 \wedge dy_2)$.

This shows that $Z(\varphi)$ is symplectically fillable, and hence by Lemma \ref{realizable if and only if nullcobordant}  it is the singular locus of $(M,\w)$ a closed, oriented $b$-symplectic manifold.
The construction of $(M,\w)$ non-orientable deferred to the end of the section.
\end{proof}

\begin{proof}[Proof of Theorem B]
 Let $(Z,\eta,\theta)$ be a $3$-dimensional cosymplectic manifold. By Lemma \ref{realizable if and only if nullcobordant} all we must show  is that $(Z,\eta,\theta)$ is symplectically fillable (again, the case of $(M,\w)$ non-orientable is deferred to the end of the section). By Lemma \ref{lem : cosymp cob}, we may assume without loss  of generality that $(Z,\eta,\theta)$ is a symplectic mapping torus. But, according to Eliashberg \cite{E04}, all $3$-dimensional symplectic mapping tori are symplectically fillable.
 \end{proof}

 \begin{remark}[A comment of Y. Mitsumatsu] One can give an alternative proof of Theorem B without using the full strength of Eliashberg's result as follows: every symplectic transformation on a closed surface $\Sigma$ is symplectically isotopic to
a word on Dehn twists \cite{Li64}. By the proof of Theorem C, our starting mapping torus is cobordant to one whose monodromy is symplectically isotopic to the identity. This means that we arrived at the trivial mapping torus $\Sigma\times \s^1$, but with
\[\eta=\sigma +\theta\wedge \beta,\]
where $\beta$ is a closed 1-form on $\Sigma$. By Lemma \ref{lem : cosymp cob}, we may assume without loss
 of generality that $\beta$ has integral periods. As Y. Mitsumatsu has pointed out to us, one can get rid of the summand $\theta\wedge \beta$ by first applying a positive Lagrangian surgery, followed by a negative Lagrangian surgery, around the Poincar\'e dual to $\beta$ (and rescaling the original $\eta$).
\end{remark}

\begin{remark}[The non-orientable case]
A few words about the non-orientable case are in order. The normal bundle to the singular locus of a non-orientable $b$-symplectic manifold is classified by a nontrivial $\tau\in H^1(Z;\Z_2)$, which determines a two-sheeted covering $Z_{\tau} \to Z$. It is not difficult to see that a cosymplectic manifold $(Z,\theta,\eta)$ is the singular locus of a non-orientable $(M,\w)$, with normal bundle given by $\tau\in H^1(Z;\Z_2)$, if and only if the covering $(Z_\tau,\theta_\tau,\eta_\tau)$ is the singular locus of an orientable $b$-symplectic manifold.

This allows one to prove version of Theorems B and C when the ambient $b$-symplectic manifold is non-orientable.
\end{remark}

\end{document}